\documentclass[a4paper,10pt]{amsart}

\usepackage[english]{babel}
\usepackage{amsmath,amssymb,amsthm,mathtools}
\usepackage{hyperref}
\usepackage[centering]{geometry}
\usepackage{color}

\newcommand{\sh}{\ensuremath{\bb{S}^2\times \bb{H}^2}}
\newcommand{\nc}[1]{\ensuremath{\nabla_{#1}^\bot}}                  
\newcommand{\ip}[3][]{\ensuremath{\langle#2,#3\rangle}}   
\newcommand{\bb}[1]{\ensuremath{\mathbb{#1}}}                        
\newcommand{\Si}{\ensuremath{\Sigma}}                                
\newcommand{\con}[1]{\ensuremath{\nabla_{#1}}}
\newcommand{\conMix}{\ensuremath{\overline{\nabla}}}
\newcommand{\conM}[1]{\ensuremath{\tilde{\nabla}_{#1}}}
\newcommand{\dz}{\ensuremath{\partial_z}}
\newcommand{\dzbar}{\ensuremath{\partial_{\bar{z}}}}
\newcommand{\dx}{\ensuremath{\partial_x}}
\newcommand{\dy}{\ensuremath{\partial_y}}
\newcommand{\zbar}{\ensuremath{\bar{z}}}
\newcommand{\norm}[1]{\left\lVert#1\right\rVert}

\renewcommand{\epsilon}{\varepsilon} 

\newcommand{\R}{\ensuremath{\mathbb{R}}}
\newcommand{\C}{\ensuremath{\mathbb{C}}}
\renewcommand{\H}{\ensuremath{\mathbb{H}}}

\newtheorem{theorem}{Theorem}[section] 
\newtheorem*{theorem*}{Theorem} 
\newtheorem{lemma}[theorem]{Lemma}
\newtheorem{proposition}[theorem]{Proposition}

\theoremstyle{definition}

\newtheorem{corollary}[theorem]{Corollary}

\newtheorem{remark}[theorem]{Remark}

\numberwithin{equation}{section}

\title[Spheres with parallel mean curvature in $\bb{S}^2\times \bb{H}^2$]{Spheres with parallel mean curvature in $\bb{S}^2\times \bb{H}^2$ }

\author{Giel Stas \and Joeri Van der Veken}

\address{G. Stas \and J. Van der Veken, KU\ Leuven, Department of Mathematics, Celestijnenlaan 200B -- Box 2400, 3001 Leuven, Belgium}
\email{giel.stas@kuleuven.be, joeri.vanderveken@kuleuven.be}
\thanks{J. Van der Veken is supported by the Research Foundation-Flanders (FWO) and the Fonds de la Recherche Scientifique (FNRS) under EOS project G0I2222N and by the KU Leuven Research Fund under project 3E210539}

\subjclass[2010]{Primary 53C42, 53A10; Secondary 53C40}
\keywords{Parallel mean curvature, surface, Riemannian product space, Hopf differential}

\date{}

\begin{document}

\begin{abstract}
It is known that a surface with parallel mean curvature vector field in a Riemannian product of two surfaces of constant Gaussian curvature carries a holomorphic quadratic differential. In this paper we consider the Riemannian product of a sphere and a hyperbolic plane of opposite Gaussian curvatures and study the parallel mean curvature surfaces for which the differential vanishes. In particular, we classify all parallel mean curvature spheres, for which the differential vanishes for topological reasons.
\end{abstract}

\maketitle

\section{Introduction}

Surfaces with \emph{constant mean curvature} (CMC) are important objects of study in geometry and analysis with several applications in science and technology. One of the main reasons for this is that they are solutions to a variational problem. Namely, closed CMC surfaces in three-dimensional ambient spaces are critical points of the area functional under the restriction that the enclosed volume is fixed. One of the first major results in the study of CMC surfaces was obtained by Heinz Hopf in 1951 \cite{Hopf}: he proved that a CMC surface in Euclidean three-space $\R^3$ with the topology of a sphere must be a round sphere. He did this by constructing a holomorphic quadratic differential on any CMC surface in $\R^3$, which must vanish if the surfaces has the topology of a sphere. Inspired by this technique, the use of holomorphic quadratic differentials is now  common in the study of CMC surfaces and they are often referred to as \emph{Hopf differentials}. 

For submanifolds of higher codimension, the mean curvature is not merely a function on the submanifold, but a normal vector field. A submanifold is called CMC if this vector field has constant length. A more restrictive condition is that the mean curvature vector field is parallel in the normal bundle. Such submanifolds are called \textit{parallel mean curvature} (PMC) submanifolds. The use of Hopf differentials turns out to be useful when studying PMC surfaces in real space forms: Ferus and Hoffman classified PMC spheres in real space forms of arbitrary dimension and arbitrary constant sectional curvature in \cite{Ferus71} and \cite{Hoffman73}. This was later generalized to a full classification of PMC surfaces in real space forms by Chen and Yau, see \cite{Chen73} and \cite{Yau74}.

More recently, several authors started studying CMC and PMC surfaces in ambient spaces of non-constant sectional curvature. A natural family of ambient spaces to start this study are the Riemannian products of two real space forms. In 2004, Abresch and Rosenberg introduced a Hopf differential for CMC surfaces in $\mathbb S^2(c) \times \R$ and $\mathbb H^2(-c) \times \R$ in \cite{AbrRos}. They also classified CMC surfaces with vanishing Hopf differentials, including the CMC spheres. Leite further discussed this topic in \cite{Leite}. A discussion of PMC surfaces in $\mathbb S^n(c) \times \R$ and $\mathbb H^n(-c) \times \R$ for $n>2$ can be found in \cite{AdCT}. In particular, Alencar, do Carmo and Tribuzy proved a reduction of codimension, which reduces the problem to a classification in $\mathbb S^4(c) \times \R$ and $\mathbb H^4(-c) \times \R$.

In 2011, Torralbo and Urbano found two Hopf differentials on any PMC surface in $\mathbb S^2(c) \times \mathbb S^2(c)$ or $\mathbb H^2(-c) \times \mathbb H^2(-c)$ which are generally linearly independent. They used these to classify the PMC surfaces for which the differentials vanish, in particular the PMC spheres \cite{TorralboUrbano}. It is essential for their approach that both factors of the ambient space have the same constant Gaussian curvature, which makes the space Einstein. However, around the same time, Kowalczyk  \cite{ThesisKowalczyk} and de Lira and Vitório \cite{LiraVitorio} independently found one single Hopf differential on any PMC surface in a product of two space forms of any dimension and any constant sectional curvatures. In \cite{LiraVitorio}, this Hopf differential is used to classify the PMC spheres of $\mathbb{S}^3(c)\times \mathbb{R}$ and $\mathbb{H}^3(c)\times \mathbb{R}$.

In the present paper, we consider the PMC surfaces in $\mathbb S^2(c) \times \mathbb H^2(-c)$ whose Hopf differential $\Theta$ vanishes. In Section \ref{trivialDifferential}, we introduce a frame which is particularly well adapted to the problem. This allows us to show in Theorem \ref{nonTrivialIntersection} that any tangent space to such a PMC surface intersects the tangent space to either the $\mathbb S^2(c)$- or the $\mathbb H^2(-c)$-component non-trivially. In particular, the following inclusions hold: \begin{multline*}
		\lbrace \text{PMC spheres} \rbrace \subseteq \lbrace \text{PMC surfaces with vanishing Hopf differential } \Theta \rbrace \\
		\subseteq \left\lbrace \begin{array}{c} \text{PMC surfaces } \Sigma \text{ such that for every }p\in\Sigma,~T_p\Sigma \text{ intersects } \\ T_{\pi_1(p)}\bb{S}^2(c)\times \lbrace 0 \rbrace \text{ or } \lbrace 0 \rbrace\times T_{\pi_2(p)}\bb{H}^2(-c) \text{ non-trivially} \end{array} \right\rbrace.
\end{multline*}
Here, $\pi_1$ and $\pi_2$ denote the projections onto the $\bb{S}^2(c)-$ and $\H^2(-c)$-component, respectively. Lemma \ref{Csquared} relates this last family to the so called \emph{K\"ahler functions}, which can be defined on any surface in $\mathbb S^2(c) \times \mathbb H^2(-c)$. From that lemma, it can be seen that the last family coincides with the family of PMC surfaces studied in \cite{TorralboUrbano} within the ambient spaces $\bb{S}^2(c)\times \bb{S}^2(c)$ and $\H^2(-c)\times \H^2(-c)$. This leads to examples of PMC surfaces in Section~\ref{examples} and to the classification of all three families of PMC surfaces mentioned above in Section~\ref{classification}. In particular, this includes the PMC surfaces with vanishing Hopf differential and the PMC spheres. The main result of the present paper is Theorem \ref{MainClassification}, which is an analog to \cite[Theorem 3]{TorralboUrbano}. It classifies the PMC surfaces $\Sigma$ such that for every $p\in\Sigma,$ the tangent space $T_p\Sigma$ intersects $T_{\pi_1(p)}\bb{S}^2(c)\times \lbrace 0 \rbrace$ or $\lbrace 0 \rbrace\times T_{\pi_2(p)}\bb{H}^2(-c)$ non-trivially, which is the last family mentioned above. Such a surface is locally congruent to one of the following:
\begin{enumerate}
\item a CMC surface in a totally geodesic $\mathbb S^2(c)\times\gamma$, where $\gamma$ is a geodesic of $\mathbb H^2(-c)$,
\item a CMC surface in a totally geodesic $\gamma\times\mathbb H^2(-c)$, where $\gamma$ is a geodesic of $\mathbb S^2(c)$,
\item a product of curves with constant curvatures in $\mathbb S^2(c)$ and $\mathbb H^2(-c)$,
\item a member of a more involved family of surfaces, see Proposition \ref{SpecialSurfaces1} and Proposition \ref{SpecialSurfaces2}.
\end{enumerate}
The PMC spheres belong to the first two families in the list above, i.e., they arise from CMC spheres in $\mathbb S^2(c) \times \R$ or $\mathbb H^2(-c) \times \R$, which are classified in \cite{AbrRos}.

It should be noted that the results in Sections \ref{examples} and \ref{classification} are similar to the corresponding results for $\mathbb S^2(c) \times \mathbb S^2(c)$ and $\mathbb H^2(-c) \times \mathbb H^2(-c)$ in \cite{TorralboUrbano}. Therefore, analogous proofs are omitted. The result on the intersections of the tangent spaces proven in Section \ref{trivialDifferential} requires vastly different techniques. This is because we only have one Hopf differential, rather than two. Furthermore, the expression for the differential as well as some formulas are slightly different. One of the main classification results in \cite{TorralboUrbano}, Theorem 3, relies in a minor way on the two Hopf differentials, too. In the proof Theorem \ref{MainClassification}, we describe a small but necessary adaptation to generalize the classification to our setting. It should also be noted that the classification of PMC spheres in a product of two surfaces of constant Gaussian curvatures which are neither equal, nor opposite, remains an open problem.

\section{Preliminaries}

\subsection{Isometric immersions and mean curvature conditions} Let $\Phi: M \to \tilde M$ be an isometric immersion between Riemannian manifolds. We will often (locally) identify $M$ with its image under $\Phi$ and denote both Riemannian metrics by $\ip{~}{~}$. The Levi-Civita connections of $M$ and $\tilde M$ are denoted by $\nabla$ and $\tilde{\nabla}$, respectively. The \emph{second fundamental form} $h$, the \emph{shape operators} $A_{\xi}$ and the \emph{normal connection} $\nabla^{\perp}$ are defined via the formulas of Gauss and Weingarten as follows. Let $X$ and $Y$ be tangent vector fields to $M$ and let $\xi$ be a normal vector field to $M$, then one can decompose both $\tilde\nabla_XY$ and $\tilde\nabla_X\xi$ in a component tangent to $M$ and a component normal to $M$:
\begin{align}
& \tilde\nabla_XY = \nabla_XY + h(X,Y), \label{eq:Gauss}\\
& \tilde\nabla_X\xi = -A_{\xi}X + \nabla^{\perp}_X\xi.
\label{eq:Weingarten}
\end{align}
The shape operators are related to the second fundamental form by
\begin{equation} \label{eq:relationAh}
\langle A_{\xi}X,Y \rangle = \langle h(X,Y),\xi \rangle. 
\end{equation}

An important extrinsic invariant of an isometric immersion $\Phi: M \to \tilde M$ is the \emph{mean curvature}
$$ H = \frac{1}{\dim M} \, \mathrm{tr}(h), $$
which is a vector field along $\Phi$ taking values in the normal bundle. The immersion $\Phi$ is called \emph{minimal} if $H$ vanishes and we say that $\Phi$ has \emph{constant mean curvature} (CMC) if $H$ has constant length. The mean curvature condition that we consider in the present paper is the following: $\Phi$ has \emph{parallel mean curvature} (PMC) if $\nc{~} H = 0$, while $H \neq 0$. For hypersurfaces, the notions of PMC and CMC are equivalent. For higher codimensions, the PMC condition implies the CMC condition, but the converse does not hold in general.

\begin{remark}
We exclude minimal immersions in the definition of PMC immersions, as certain techniques we use later on do not apply to them. This is a common convention in the literature.
\end{remark}

The following proposition is a classical result, a proof of which can for example be found in \cite[Proposition 3.1 and Remark 3.2]{PMC}. It can be used to generate examples of PMC surfaces and various classification results build on this statement.

\begin{proposition}\label{TrivialPMC}
Let $\tilde M$ be a four-dimensional manifold and $M$ a totally umbilical CMC hypersurface of $\tilde M$. Then every CMC surface immersed in $M$ is either a PMC surface or a minimal surface in $\tilde M$. The latter only occurs if the surface is minimal in $M$ and $M$ is totally geodesic in $\tilde M$.
\end{proposition}

Note that an immersion is \emph{totally geodesic} if all its shape operators, and hence its second fundamental form, vanish. More generally, it is \emph{totally umbilical} if all its shape operators are scalar multiples of the identity. Further on, we will use \textit{pseudo-umbilical} immersions, defined by the property that the shape operator associated to the mean curvature, $A_H$, is a multiple of the identity. In that case, one necessarily has $A_H = \norm{H}^2 \text{id}$.

\subsection{The product space $\bb{S}^2\times \bb{H}^2$} In this paper, we consider PMC surfaces in $\mathbb S^2(c) \times \mathbb H^2(-c)$, the Riemannian product of a 2-sphere and a hyperbolic plane with opposite Gaussian curvatures. Since the PMC condition is invariant under homotheties of the ambient space, we may assume that $c=1$. For brevity, we will denote $\mathbb S^2(1)$ and $\mathbb H^2(-1)$ by $\mathbb S^2$ and $\mathbb H^2$ respectively.

In reference to Proposition \ref{TrivialPMC}, it was shown in \cite[Proposition 1]{TorralboUrbano} that a totally umbilical CMC hypersurface of $\mathbb S^2 \times \mathbb S^2$ or $\mathbb H^2 \times \mathbb H^2$ must be  totally geodesic. The following proposition is the corresponding result for $\sh$. Since its proof is similar to the one given in \cite{TorralboUrbano}, we omit it here.

\begin{proposition}\label{CMCtu}
Let $\Phi\colon M \to \sh$ be a totally umbilical CMC hypersurface. Then $\Phi$ is totally geodesic and it is locally congruent to one of the following:
\begin{enumerate}
\item $\Phi_1: \bb{S}^2 \times \bb{R} \to \mathbb S^2 \times \mathbb H^2: (p,t) \mapsto (p,\gamma(t))$, the product of $\bb{S}^2$ and a geodesic $\gamma$ of $\bb{H}^2$,
\item $\Phi_2: \bb{H}^2 \times \bb{R} \to \mathbb S^2 \times \mathbb H^2: (p,t) \mapsto (\gamma(t),p)$, the product of a geodesic $\gamma$ of $\bb{S}^2$ and $\bb{H}^2$.
\end{enumerate}
\end{proposition}

We have thus found a first type of PMC surface in $\sh$: non-minimal CMC surfaces in $\bb{S}^2\times \bb{R}$ and $\bb{H}^2 \times \bb{R}$ give rise to PMC surfaces in $\sh$ by composing their immersions with $\Phi_1$ or $\Phi_2$. For a classification of CMC surfaces in $\bb{S}^2\times \bb{R}$ and $\bb{H}^2 \times \bb{R}$, see \cite{AbrRos}.

Let $\pi_1: \sh \to \mathbb S^2$ and $\pi_2: \sh \to \mathbb H^2$ denote the projections on the two factors. Any point $p\in \sh$ may be written as $p=(\pi_1(p),\pi_2(p))$. The tangent space can be identified with a product of tangent spaces, namely $T_p(\sh)\cong T_{\pi_1(p)}\bb{S}^2 \times T_{\pi_2(p)} \bb{H}^2$. Using this identification, the \emph{product structure} of $\sh$ is the $(1,1)$-tensor field $F$ defined by
\begin{equation} \label{eq:defF}
F : T(\sh) \to T(\sh): (v_1,v_2)_p \mapsto (v_1,-v_2)_p.
\end{equation}
Note that $F$ is symmetric, squares to the identity and is parallel with respect to the Levi-Civita connection.
Using $F$, one has the following expression for the curvature tensor of $\sh$, see for example \cite[equation (2.36)]{ThesisKowalczyk}.

\begin{lemma}
The Riemann-Christoffel curvature tensor $\tilde R$ of $\sh$ is given by
\begin{equation}\label{eq:Rtilde}
\tilde R(X,Y) = \frac{X+FX}{2} \wedge \frac{Y+FY}{2} - \frac{X-FX}{2} \wedge \frac{Y-FY}{2}, \end{equation}
where $(A \wedge B)C = \langle B,C \rangle A - \langle A,C \rangle B$ for all tangent vectors $A$, $B$ and $C$ to $\sh$.
\end{lemma}

It is easy to see that, at any point of $\sh$, the derivatives are $d\pi_1 = (\mathrm{id}+F)/2$ and $d\pi_2 = (\mathrm{id}-F)/2$. Therefore, the following alternative expressions for $\tilde R$ follow immediately from \eqref{eq:Rtilde}:
\begin{align*}\label{LTVCurvature}
& \tilde R(X,Y) = (d\pi_1)X \wedge (d\pi_1)Y - (d\pi_2)X \wedge (d\pi_2)Y, \\
& \tilde R(X,Y) = - X \wedge Y + (d\pi_1)X \wedge Y + X \wedge (d\pi_1) Y , \\
& \tilde R(X,Y) = X \wedge Y - (d\pi_2)X \wedge Y - X \wedge (d\pi_2)Y .
\end{align*}
The last expression can be found in \cite[equation (2.1)]{LTV}.

With these formulas for $\tilde R$, one can write the Codazzi and Ricci equations for isometric immersions into $\sh$ as follows, see \cite[Section 2.3]{ThesisKowalczyk}.

\begin{lemma}\label{Codazzi}
Let $\Phi: M \to \sh$ be an isometric immersion with second fundamental form~$h$. Then, for all vectors $X$, $Y$ and $Z$ tangent to $M$,
\begin{equation}\label{eq:codazzi}
(\conMix h)(X,Y,Z) - (\conMix h)(Y,X,Z) = \frac 12 \left( F(X\wedge Y)Z \right)^\bot, \end{equation}
where $(\overline{\nabla}h)(A,B,C)=\nabla^{\perp}_Ah(B,C) - h(\nabla_AB,C) - h(B,\nabla_AC)$ for all vectors $A$, $B$ and $C$ tangent to $M$ and the superscript $\perp$ on the right hand side of \eqref{eq:codazzi} denotes the component normal to $M$.
\end{lemma}

\begin{lemma}\label{Ricci}
Let $\Phi: M \to \sh$ be an isometric immersion. Assume that $X$ and $Y$ are tangent vectors to $M$ and that $\xi$ and $\eta$ are normal vectors to $M$. Then, the normal curvature tensor $R^\bot$, i.e., the curvature tensor associated to the normal connection $\nabla^{\perp}$, satisfies
\begin{equation}
\langle R^\bot(X,Y)\xi,\eta\rangle = \langle [A_\xi,A_\eta]X,Y\rangle.
\end{equation}
In particular, if $\Phi\colon\Si\to \sh$ is a PMC surface, then $[A_H,A_\eta] = 0$ for any normal vector $\eta$.
\end{lemma}

\subsection{Conformal coordinates and a Hopf differential}\label{subsec:complexcoord} Conformal coordinates $(x,y)$ on a Riemannian surface $\Sigma$ are coordinates such that the metric takes the form $e^{2u}(dx^2 + dy^2)$ for some smooth function $u$ on $\Sigma$. It is a classical result that such coordinates locally always exist on smooth Riemannian surfaces. If $\Sigma$ is oriented, we can moreover impose that $(\dx,\dy)$ is positively oriented and introduce a local complex parameter $z = x+iy$. Since the transition map between two such parameters is holomorphic, this gives $\Si$ the structure of a Riemann surface. Using the complex vectors $\dz = (\dx-i\dy)/2$ and $\dzbar = (\dx+i\dy)/2$, the geometry of $\Sigma$ can be described as follows, see for example \cite[Lemma 4.1]{PMC}.

\begin{lemma}\label{MTVdVLemma}
In the setting described above, the complex linear extension of the metric on $\Si$ is given by
\begin{equation} 
\ip{\dz}{\dzbar} = \frac{1}{2}e^{2u} \text{ and } \ip{\dz}{\dz} = \ip{\dzbar}{\dzbar} = 0.
\end{equation}
The Levi-Civita connection of $\Si$ is determined by
\begin{equation} \label{eq:nabladz}
\nabla_{\dz}\dz =  2 u_z \dz, \ \nabla_{\dzbar}\dzbar =  2 u_{\bar z} \dzbar \text{ and } \con{\dz}\dzbar = \con{\dzbar}\dz = 0. \end{equation}
If $\Phi\colon\Si \to M$ is an immersion of $\Si$ into an arbitrary Riemannian manifold $M$, then the second fundamental form $h$ satisfies
\begin{equation} \label{eq:hdzdzbar}
h(\dz,\dzbar) = \frac 12 e^{2u}H,
\end{equation}
where $H$ is the mean curvature vector field.
\end{lemma}

Now assume that $\Phi: \Sigma \to M$ is a PMC immersion of a surface with local conformal coordinates into a product of two real space forms. Using these conformal coordinates, one can define a holomorphic quadratic differential, see \cite{ThesisKowalczyk} and \cite{LiraVitorio}. Interpreting the expressions for this differential for $M=\sh$, we find the following.

\begin{lemma}\label{Hopf}
Let $\Phi\colon\Si\to \sh$ be a PMC immersion of a surface $\Si$ into $\sh$. Then there is a well-defined holomorphic quadratic differential $\Theta$ such that with respect to any local conformal coordinate $z$, $\Theta$ can be expressed as
\begin{equation} \label{eq:Hopf}
\Theta = \left(4\ip{h(\dz,\dz)}{H} + \ip{F\dz}{\dz}\right) dz\otimes dz.
\end{equation}
\end{lemma}
For reasons explained in the introduction of this paper, we refer to $\Theta$ as a \emph{Hopf differential}. It is well-defined, since $dz$ and $\dz$ transform inversely under a change of conformal coordinates. To check that it is indeed holomorphic, it suffices to verify that its $\dzbar$-derivative vanishes, using Lemma \ref{MTVdVLemma}.

\subsection{K\"ahler functions}
The later sections will use similar techniques as in \cite{TorralboUrbano}. Therefore, this subsection concerns the necessary preliminaries on K\"ahler functions and the associated Frenet system.

Suppose that $\Phi\colon\Sigma\to \sh$ is a PMC surface. Denote the standard complex structures on $\bb{S}^2$ and $\bb{H}^2$ by $J_{\mathbb{S}^2}$ and $J_{\mathbb{H}^2}$ and consider the following complex structures on $\sh$:
\begin{equation} J_1 = (J_{\mathbb{S}^2},J_{\mathbb{H}^2}), \qquad J_2 = (J_{\mathbb{S}^2},-J_{\mathbb{H}^2}). \end{equation}
Since $J_{\mathbb{S}^2}$ and $J_{\mathbb{H}^2}$ are parallel, so are $J_1$ and $J_2$. Thus, these are K\"ahler. Let $\omega_j$ denote the corresponding K\"ahler 2-forms, that is, $\omega_j(X,Y) = \ip{J_j X}{Y}$, and let $\omega_{\Sigma}$ denote the area 2-form on $\Sigma$. Then the K\"ahler functions $C_1,C_2$ on $\Sigma$ are defined by
\begin{equation}\label{eq:defKahler} \Phi^*\omega_j = C_j \omega_{\Sigma}, \quad j\in \lbrace 1,2\rbrace.\end{equation}
Let $z=x+iy$ be a complex coordinate on $\Sigma$ with conformal factor $e^{2u}$ as in subsectionH \ref{subsec:complexcoord}. Denoting $\eta = H/\|H\|$ and using the standard orientation on $\sh$, there exists a unique normal vector field $\tilde \eta$ such that $( e^{-u}\partial_x, e^{-u}\partial_y, \tilde \eta, \eta )$ is a positively oriented orthonormal frame for $T (\sh)$. Define the complexified normal vector field
\begin{equation}\label{eq:defxi} \xi = \frac{1}{\sqrt{2}}(\eta-i\tilde{\eta}).\end{equation}

As in \cite{TorralboUrbano}, a Frenet system can be obtained for PMC surfaces of $\sh$. Using the standard model for $\bb{S}^2$ and the hyperboloid model for $\bb{H}^2$, $\sh$ can be seen as a subspace of $\R^3\times \R^3_1 \cong \R^6_1.$ The following proposition summarizes the fundamental data and the Frenet system arising from this immersion into $\R^6_1$. There are minor differences with \cite[Section 3]{TorralboUrbano}.

\begin{proposition}\label{TUsummary}
	Let $\varphi \colon \Sigma \to \sh$ be a PMC surface of $\sh \subseteq \R^6_1$. Denote the immersion of $\sh$ into $\R^6_1$ by $\iota\colon \sh \to \R^6_1$ and define $\Phi = (\phi,\psi) = \iota \circ \varphi \colon \Sigma \to \R^6_1$. Let $z$ be a conformal coordinate on a neighborhood $U \subseteq \Sigma$ and let $\xi$ denote the complexified normal vector field introduced in equation \eqref{eq:defxi}. Let $C_1$ and $C_2$ be the K\"ahler functions defined in \eqref{eq:defKahler}. Then $\Phi$ and $\hat{\Phi} = (\phi,-\psi)$ form a reference frame for the normal space of $\sh$ with
	\begin{equation}\label{eq:Phiproduct} \ip{\Phi}{\Phi} = \langle\hat{\Phi},\hat{\Phi}\rangle = 0, \qquad \langle\Phi,\hat{\Phi}\rangle = 2.\end{equation}
	There exist functions $\gamma_1,\gamma_2,f_1,f_2\colon U \to \C$ such that
	\begin{equation}\label{eq:gammanorm}
		|\gamma_1|^2 = \frac{e^{2u}}{2} (1-C_1^2), \quad |\gamma_2|^2 = \frac{e^{2u}}{2} (1-C_2^2)
	\end{equation}
	and such that the frame $(\partial_z, \partial_{\zbar}, \xi, \bar{\xi},\Phi,\hat{\Phi} )$ satisfies
	\begin{equation}\label{eq:frenet}
		\begin{cases}
			D_{\dz}\dz = 2\dz u\dz + f_1\xi + f_2 \overline{\xi} - \frac{1}{2}\gamma_1\gamma_2 \Phi, \\
			D_{\dzbar} \dz = \frac{\norm{H}e^{2u}}{2\sqrt{2}}(\xi + \overline{\xi}) - \frac{e^{2u}}{4}C_1C_2 \Phi - \frac{e^{2u}}{4} \hat{\Phi}, \\
			D_{\dz} \xi = -\frac{\norm{H}}{\sqrt{2}} \dz - 2e^{-2u} f_2 \dzbar + \frac{i}{2} C_1 \gamma_2 \Phi, \\
			D_{\dz} \overline{\xi} = -\frac{\norm{H}}{\sqrt{2}} \dz - 2e^{-2u} f_1 \dzbar + \frac{i}{2} C_2 \gamma_1 \Phi,
		\end{cases}
	\end{equation}
	where $D$ denotes the Levi-Civita connection of $\R^6_1$. The functions $C_j,\gamma_j,$ and $f_j$ for $j \in \lbrace 1,2\rbrace$ satisfy the differential equations
	\begin{equation} \label{eq:DCj&Dgammaj}
		\dz C_j = 2ie^{-2u}\overline{\gamma_j}f_j - i \frac{\norm{H}}{\sqrt{2}} \gamma_j, \quad \dzbar\gamma_j = -\frac{i \norm{H}e^{2u}C_j}{\sqrt{2}}, \quad \dz\gamma_j = 2\dz u\gamma_j  - 2iC_jf_j,
	\end{equation}
	\begin{equation} \label{eq:Dfj}
		\dzbar f_1 = i \frac{e^{2u}}{4} C_2 \gamma_1 \qquad \text{and} \qquad \dzbar f_2 = i \frac{e^{2u}}{4} C_1 \gamma_2.
	\end{equation}
\end{proposition}
\begin{proof}
		The methods to obtain these formulas are the same as in Section 3 of \cite{TorralboUrbano}. The equations of the Frenet system \eqref{eq:frenet} are different in the $\Phi$- and $\hat{\Phi}$-components because of the inner products \eqref{eq:Phiproduct}. Namely, in \cite{TorralboUrbano}, $\norm{\Phi}^2 = \|\hat{\Phi}\|^2 = \pm 2$ and $\ip{\Phi}{\hat{\Phi}} = 0$, whereas, in the present paper, $\norm{\Phi}^2 = \|\hat{\Phi}\|^2 = 0$ and $\ip{\Phi}{\hat{\Phi}} = 2$. This change also leads to different expressions for $\dzbar f_1$ and $\dzbar f_2$ in \eqref{eq:Dfj}. 
 \end{proof}
\begin{remark}
	The Hopf differential \eqref{eq:Hopf} can be written in terms of the functions in Proposition \ref{TUsummary} as
	\begin{equation}\label{eq:HopfAlternative} \Theta = (2\sqrt{2} \norm{H} (f_1+f_2) + \gamma_1 \gamma_2) dz\otimes dz. \end{equation}
	The first term of the differential \eqref{eq:Hopf}, $4\ip{h(\dz,\dz)}{H}$, can be retrieved from the first equation of the Frenet system \eqref{eq:frenet}. By the definition of $\xi$ \eqref{eq:defxi}, this term equals $4\ip{h(\dz,\dz)}{H} = 2\sqrt{2}\norm{H}(f_1+f_2).$ To rewrite $\ip{F\dz}{\dz}$, note that $F = -J_1J_2$. From equations (3.2) and (3.3) in \cite{TorralboUrbano}, one obtains $\ip{F\dz}{\dz} = \ip{J_1\dz}{J_2\dz} = \gamma_1\gamma_2.$
\end{remark}

\section{PMC surfaces in $\sh$ with vanishing Hopf differential}\label{trivialDifferential}

The goal of this section is to prove Theorem \ref{nonTrivialIntersection}, which states that if the Hopf differential \eqref{eq:Hopf} vanishes identically on a PMC surface in $\sh$, then every tangent space to that surface contains a non-zero vector that is either tangent to $\mathbb S^2$ or to $\mathbb H^2$. We will need several lemmas to achieve this.

\begin{lemma}\label{ShapeOperatorH}
Let $\Phi: \Sigma \to \sh$ be a PMC surface whose  Hopf differential \eqref{eq:Hopf} vanishes. Let $f$ be the symmetric $(1,1)$-tensor field
\begin{equation}\label{eq:deff}
f: T\Sigma \to T\Sigma\colon v \mapsto (Fv)^\top,
\end{equation}
where the superscript $\top$ denotes the component tangent to $\Sigma$, and let $\mathrm{Adj}(f)$ be its adjugate. Then the shape operator associated to the mean curvature vector field $H$ is
\begin{equation}\label{eq:AH} 
A_H = \norm{H}^2 \mathrm{id} - \frac{1}{8}f + \frac{1}{8}\mathrm{Adj}(f).
\end{equation}
\end{lemma}

\begin{proof}
Since both sides of the equation are symmetric tensors, it suffices to show that the tensors yield the same image for $ \dx$ and $\dy$, where $z = x+iy$ is a conformal parameter. When the Hopf differential \eqref{eq:Hopf} vanishes the equation $4\ip{h(\dz,\dz)}{H} = -\ip{F\dz}{\dz} = -\ip{f\dz}{\dz}$ holds. Denote the matrix of $f$ with respect to the basis $( \dx, \dy )$ as
	$$f = \begin{pmatrix}a&b\\b&c\end{pmatrix},$$
	where $a,b,c$ are smooth functions.
	Then the two terms of the differential \eqref{eq:Hopf} can be expressed as
		$-\ip{f\dz}{\dz} = -\frac{e^{2u}}{4}(a-2ib-c)$ and
		$4\ip{h(\dz,\dz)}{H} = \ip{h(\dx,\dx) - h(\dy,\dy)}{H} - 2i\ip{h(\dx,\dy)}{H}$.
	Thus, the equation $4\ip{h(\dz,\dz)}{H} = -\ip{f\dz}{\dz}$ can be split into its real and imaginary part to obtain two real-valued equations. From the definition of the mean curvature vector, there is an additional equation
	$\ip{h(\dx,\dx)+h(\dy,\dy)}{H} = 2e^{2u}\ip{H}{H} = 2e^{2u}\norm{H}^2.$
	This yields a system of three linear equations for $\ip{h(\dx,\dx)}{H}, \ip{h(\dx,\dy)}{H}$ and $\ip{h(\dy,\dy)}{H}$. The unique solution is
	$$\ip{h(\dx,\dx)}{H} = e^{2u}\norm{H}^2-\frac{e^{2u}(a-c)}{8}, \qquad \ip{h(\dy,\dy)}{H} = e^{2u}\norm{H}^2+\frac{e^{2u}(a-c)}{8},$$
	$$\ip{h(\dx,\dy)}{H} = -b\frac{e^{2u}}{4}. $$
	The metric satisfies $\ip{\dx}{\dx} = \ip{\dy}{\dy} = e^{2u}, \ip{\dx}{\dy} = 0$, so by equation \eqref{eq:relationAh} the matrix of $A_H$ with respect to $(\dx,\dy)$ is
	$$A_H = \begin{pmatrix}
		\norm{H}^2-\frac{a}{8} + \frac{c}{8} & -\frac{b}{4} \\ -\frac{b}{4} & \norm{H}^2+\frac{a}{8}  -\frac{c}{8} 
	\end{pmatrix}= \norm{H}^2\text{Id} - \frac{1}{8}f + \frac{1}{8}\mathrm{Adj}(f).$$
\end{proof}

It follows from \eqref{eq:AH} that a local  frame on $\Sigma$ diagonalizing $f$ also diagonalizes $A_H$. Moreover, by Lemma~\ref{Ricci}, $A_H$ commutes with all shape operators. Thus, a frame diagonalizing $f$ diagonalizes all shape operators, unless in a pseudo-umbilical point, i.e., a point where $A_H$ is a multiple of the identity. The next lemma implies that the pseudo-umbilical points are scarce: they form a closed subset of the surface with empty interior.
\begin{lemma}\label{noUmbilics}
Let $\Phi: \Sigma \to \sh$ be a PMC surface whose  Hopf differential \eqref{eq:Hopf} vanishes. Denote by $\Sigma_0 \subseteq \Sigma$ the set of points where $\Phi$ is not pseudo-umbilical. Then $\Sigma_0$ is an open and dense subset of $\Sigma$.
\end{lemma}

\begin{proof}
First note that if $A_H$ is a multiple of the identity, it necessarily is $A_H=\|H\|^2\mathrm{id}$. Towards a contradiction, assume that there is an open subset $U \subseteq \Sigma$ where $A_H  = \norm{H}^2\mathrm{id}$. By restricting $U$ if necessary, one can assume that there is a conformal parameter $z = x+iy$ on~$U$. The first term of the Hopf differential \eqref{eq:Hopf} vanishes, as $\ip{h(\dz,\dz)}{H}  = \langle A_H\dz,\dz \rangle = \|H\|^2 \langle \dz,\dz \rangle = 0$. Since the differential vanishes, this implies $\ip{F\dz}{\dz} = 0$, which is equivalent to $\langle F\dx,\dx\rangle = \ip{F\dy}{\dy}$ and $\ip{F\dx}{\dy} = 0$. Moreover, from $\ip{F\dz}{\dz} = 0$, the fact that $F$ is parallel and symmetric and by using \eqref{eq:Gauss}, \eqref{eq:nabladz} and \eqref{eq:hdzdzbar}, one calculates that $0 = \dzbar \ip{F\dz}{\dz} = 2 \langle F\conM{\dzbar}\dz,\dz\rangle = e^{2u}\langle FH,\dz \rangle$. Therefore, $\ip{F\dx}{H} = \ip{F\dy}{H} = 0$. In summary,
$$ F\dx = a\dx + \eta_1, \qquad F\dy = a\dy + \eta_2, $$
where $a$ is a function and $\eta_1$ and $\eta_2$ are normal vector fields, orthogonal to $H$, which implies that they are linearly dependent. Since $0 = \ip{\dx}{\dy} = \ip{F\dx}{F\dy}$, at least one of the vector fields $\eta_1$ and $\eta_2$ vanishes. Because $F$ preserves the norm, it follows that $a = \pm 1$, and thus both $\eta_1$ and $\eta_2$ vanish on $U$. This means that, at every point $p$ of $U$, $T_p\Sigma$ is the eigenspace of $F_p$ with eigenvalue $1$ or $-1$. These eigenspaces are precisely $T_{\pi_1(p)}\bb{S}^2\times \lbrace 0 \rbrace$ and $\lbrace 0 \rbrace\times T_{\pi_2(p)}\bb{H}^2$. Therefore, the embedding of $U$ is part of a surface $\bb{S}^2 \times \lbrace \pi_2(p) \rbrace$ or $\lbrace \pi_1(p) \rbrace\times\bb{H}^2$. Thus, $U$ is totally geodesic in $\sh$ and therefore minimal, not PMC. This is a contradiction.
\end{proof}
One can now construct an orthonormal frame along a PMC surface in $\sh$ with vanishing Hopf differential, adapted to the product structure $F$. More precisely, the frame can be defined in a neighborhood of any point of the open and dense subset introduced in Lemma \ref{noUmbilics}.

\begin{lemma}\label{FrameLemma}
Let $\Phi: \Sigma \to \sh$ be a PMC surface whose Hopf differential \eqref{eq:Hopf} vanishes. Denote by $\Sigma_0 \subseteq \Sigma$ the open and dense subset of points where $\Phi$ is not pseudo-umbilical. Then, in a neighborhood of any point of $\Sigma_0$, there exists an orthonormal frame $(X_1,X_2,\xi_1,\xi_2)$ where $X_1,X_2$ are tangent to $\Sigma$ and $\xi_1,\xi_2$ are normal to $\Sigma$ such that the matrix of $F$ with respect to this frame is of the form
\begin{equation} \label{eq:matrixF}
F = 
\begin{pmatrix} \cos\alpha & 0 & \sin\alpha & 0 \\
0 & \cos\beta & 0 & \sin\beta \\
\sin\alpha & 0 & -\cos\alpha & 0 \\
0 & \sin\beta & 0 & -\cos\beta \end{pmatrix},
\end{equation}
where $\alpha$ and $\beta$ are functions on that neighborhood, such that $\cos\alpha\!-\!\cos\beta$ nowhere vanishes.
\end{lemma}

\begin{proof}
Fix a point $p_0 \in \Sigma_0$. By \eqref{eq:AH}, $f$ is not a multiple of the identity in a neighborhood $U_1$ of $p_0$. Since $f$ is symmetric, there are orthonormal vector fields $X_1$ and $X_2$ on $U_1$ which are eigenvectors with different eigenvalues of $f$ in all points of $U_1$. Similarly, one can construct an orthonormal frame $(\xi_1,\xi_2)$ for the normal bundle of $\Sigma$, which, at every point of a neigbourhood $U_2$ of $p_0$, consists of two eigenvectors of the symmetric $(1,1)$-tensor field
$ f^{\perp}: T^{\perp}\Sigma \to T^{\perp}\Sigma: v \mapsto (Fv)^{\perp}, $
where the last superscript $\perp$ denotes the component normal to $\Sigma$.

Hence, with respect to the orthonormal frame $(X_1,X_2,\xi_1,\xi_2)$, defined on the neighbourhood $U=U_1 \cap U_2$ of $p_0$, the matrix of $F$ is
$$ F = 
\begin{pmatrix} a & 0 & b & c \\ 
0 & d & e & j \\ 
b & e & k & 0 \\ 
c & j & 0 & l \end{pmatrix}, \ \ a\neq d. $$
Since $F$ is orthogonal, one has
$$ F = 
\begin{pmatrix} \cos\alpha & 0 & \sin\alpha & 0 \\
0 & \cos\beta & 0 & \sin\beta \\ 
\sin\alpha & 0 & -\cos\alpha & 0 \\ 
0 & \sin\beta & 0 & -\cos\beta \end{pmatrix} ~\text{or}~ 
F = 
\begin{pmatrix} \cos\alpha & 0 & 0 & \sin\alpha \\
0 & \cos\beta & \sin\beta & 0 \\ 
0 & \sin\beta & -\cos\beta & 0 \\ 
\sin\alpha & 0 & 0 & -\cos\alpha \end{pmatrix}, $$
with $\cos\alpha \neq \cos\beta$. Note that the eigenvalues of $f^{\perp}$ are opposite to the eigenvalues of $f$. By changing the order of $\xi_1$ and $\xi_2$ if necessary, one can assume that the eigenvalue of $\xi_1$, respectively $\xi_2$, for $f^{\perp}$ is opposite to the eigenvalue of $X_1$, respectively $X_2$, for $f$. Then the matrix of $F$ has the desired form.
\end{proof}

The following lemma expresses the derivatives of the functions $\alpha$ and $\beta$ introduced in Lemma~\ref{FrameLemma} and the Levi Civita connection of the surface in terms of its second fundamental form.

\begin{lemma}\label{differentialEq}
Let $\Phi: \Sigma \to \sh$ be a PMC surface whose Hopf differential \eqref{eq:Hopf} vanishes. Using the notations of Lemma \ref{FrameLemma}, the functions $\alpha$ and $\beta$ satisfy the following system of PDEs:

\begin{align}
& X_1(\alpha) = -2 \langle h(X_1,X_1), \xi_1 \rangle, & & X_2(\alpha) = 0, \label{eq:Xalpha} \\
& X_1(\beta) = 0, & & X_2(\beta) = -2 \langle h(X_2,X_2), \xi_2 \rangle. \label{eq:Xbeta}
\end{align}
Moreover, the Levi Civita connection of $\Sigma$ is locally completely determined by
\begin{equation} \label{eq:nablaX1X2}
\langle \nabla_{X_1}X_1, X_2 \rangle \!=\! \frac{\sin\beta}{\cos\alpha\!-\!\cos\beta} \langle h(X_1,X_1),\xi_2 \rangle, \ \langle \nabla_{X_2}X_1, X_2 \rangle \!=\! \frac{\sin\alpha}{\cos\alpha\!-\!\cos\beta} \langle h(X_2,X_2),\xi_1 \rangle.
\end{equation}
\end{lemma}

\begin{proof}
Recall that the product structure $F$ is parallel on $\sh$, i.e., that $\tilde\nabla_X FY - F \tilde\nabla_XY =~0$ for all vector fields $X$ and $Y$ on $\sh$, where $\tilde\nabla$ is the Levi Civita connection of $\sh$. Equations \eqref{eq:Xalpha}--\eqref{eq:nablaX1X2} follow by taking $X, Y \in \{X_1,X_2\}$ and using \eqref{eq:Gauss}, \eqref{eq:Weingarten} and \eqref{eq:matrixF}. Moreover, we used that $\|X_1\|^2 = \|X_2\|^2 = 1$ and that $h(X_1,X_2) = 0$, by Lemma \ref{ShapeOperatorH}.
\end{proof}

One can now describe the second fundamental form of a PMC surface in $\sh$ with vanishing Hopf differential with respect to $X_1$ and $X_2$.

\begin{lemma} \label{lem:h}
Let $\Phi: \Sigma \to \sh$ be a PMC surface whose Hopf differential \eqref{eq:Hopf} vanishes. Using the notations of Lemma \ref{FrameLemma}, the frame $(X_1,X_2)$ diagonalizes all shape operators of $\Phi$. In particular,
\begin{equation} \label{eq:matrixAH}
A_H = \begin{pmatrix} 
\norm{H}^2 - \frac{\cos\alpha-\cos\beta}{8} & 0 \\ 0 & \norm{H}^2 + \frac{\cos\alpha-\cos\beta}{8} 
\end{pmatrix}
\end{equation}
with respect to $(X_1,X_2)$. Define a local function $\gamma$ such that $H/\|H\| = (\cos\gamma) \xi_1 + (\sin\gamma) \xi_2$. Then $\tilde H$, defined by $\tilde H/\|H\| = -(\sin\gamma) \xi_1 + (\cos\gamma) \xi_2$, is a normal vector field with length $\|H\|$, perpendicular to $H$, and there exists a local function $\nu$ such that
\begin{equation}\label{eq:matrixHtilde} 
A_{\tilde H} = \begin{pmatrix} \nu \frac{\cos\alpha-\cos\beta}{8} & 0 \\ 
0 & -\nu \frac{\cos\alpha-\cos\beta}{8} 
\end{pmatrix} 
\end{equation}
with respect to $(X_1,X_2)$. The function $\gamma$ satisfies the PDEs
\begin{equation} \label{eq:Xgamma}
\begin{aligned}
& X_1(\gamma) = \frac{-\sin\alpha}{\cos\alpha\!-\!\cos\beta} \langle h(X_1,X_1),\xi_2 \rangle, & & X_2(\gamma) = \frac{-\sin\beta}{\cos\alpha\!-\!\cos\beta} \langle h(X_2,X_2),\xi_1 \rangle,
\end{aligned}
\end{equation}
whereas the function $\nu$ satisfies the PDEs
\begin{equation} \label{eq:Xnu}
\begin{aligned}
& X_1(\nu) = \frac{4\norm{H} \sin\alpha (\sin\gamma - \nu\cos\gamma)}{\cos\alpha-\cos\beta}, 
& & X_2(\nu) = \frac{4\norm{H} \sin\beta (\cos\gamma + \nu\sin\gamma)}{\cos\alpha-\cos\beta}.
\end{aligned}
\end{equation}
\end{lemma}

\begin{proof}
It follows from \eqref{eq:AH} and \eqref{eq:matrixF} that $A_H$ takes the form \eqref{eq:matrixAH} with respect to $(X_1,X_2)$. In particular, $(X_1,X_2)$ diagonalizes $A_H$. Since Lemma \ref{Ricci} implies that all shape operators commute, $(X_1,X_2)$ diagonalizes all shape operators. Moreover, since $\tilde H$ is perpendicular to $H$, the shape operator $A_{\tilde H}$ is traceless. Therefore, it takes the form \eqref{eq:matrixHtilde} for a local function $\nu$.

To obtain the PDEs \eqref{eq:Xgamma}, first note that $\nabla^{\perp}H=0$ implies that $\nabla^{\perp}_X\xi_1 = -X(\gamma)\xi_2$ and $\nabla^{\perp}_X\xi_2 = X(\gamma)\xi_1$ for all vectors $X$ tangent to $\Sigma$. Now use the parallelism of the product structure $F$ on $\sh$, namely $\tilde\nabla_X FY - F \tilde\nabla_XY =~0$ for all vector fields $X$ and $Y$ on $\sh$, where $\tilde\nabla$ is the Levi Civita connection of $\sh$. More precisely, take $(X,Y)=(X_1,\xi_2)$ and $(X,Y)=(X_2,\xi_1)$ and use \eqref{eq:Gauss}, \eqref{eq:Weingarten}, \eqref{eq:Xalpha}, \eqref{eq:Xbeta} and the fact that $h(X_1,X_2)=0$.

Finally, the PDEs \eqref{eq:Xnu} follow from Lemma \ref{Codazzi}, the Codazzi equation, as follows. Choose $X=Z=X_2$ and $Y=X_1$ in \eqref{eq:codazzi}, and take the inner product with $\tilde H$. The left hand side becomes
$$ \langle (\overline{\nabla}h)(X_2,X_1,X_2)-(\overline{\nabla}h)(X_1,X_2,X_2) , \tilde H \rangle = \frac{\norm{H}}{2}\nu\sin\alpha\cos\gamma + \frac{\cos\alpha-\cos\beta}{8}X_1(\nu), $$
by using that $\tilde H$ is parallel in the normal bundle, that $(X_1,X_2)$ is an orthonormal frame satisfying $h(X_1,X_2) = 0$, and \eqref{eq:Xalpha}, \eqref{eq:Xbeta} and \eqref{eq:nablaX1X2}. On the other hand, the inner product with $\tilde H$ of the right hand side is
$$ \frac 12 \langle F(X_1\wedge X_2)X_1 , \tilde H \rangle = -\frac{1}{2} \langle (\sin\alpha)\xi_1, \tilde H \rangle = \frac{\norm{H}}{2}\sin\alpha\sin\gamma. $$
Equating both expressions yields the formula for $X_1(\nu)$ in the statement. The expression for $X_2(\nu)$ can be found analogously, by choosing $X = Z = X_1$ and $Y = X_2$ in \eqref{eq:codazzi}. 
\end{proof}

\begin{remark}
Equations \eqref{eq:matrixAH} and \eqref{eq:matrixHtilde} determine the second fundamental form with respect to the normal frame $(H,\tilde H)$. Using the relation between this frame and the frame $(\xi_1,\xi_2)$, one can rewrite the terms $\langle h(X_i,X_i),\xi_j \rangle$ (which would allow to make the PDEs \eqref{eq:Xalpha}, \eqref{eq:Xbeta} and \eqref{eq:Xgamma} more explicit). Since we will use these alternative expressions in the rest of this section, we give them here:
\begin{equation} \label{eq:hXiXixij}
\begin{aligned}
& \langle h(X_1,X_1),\xi_1 \rangle = \|H\|\cos\gamma - \frac{1}{8\|H\|}(\cos\gamma+\nu\sin\gamma)(\cos\alpha-\cos\beta), \\
& \langle h(X_1,X_1),\xi_2 \rangle = \|H\|\sin\gamma - \frac{1}{8\|H\|}(\sin\gamma-\nu\cos\gamma)(\cos\alpha-\cos\beta),\\
& \langle h(X_2,X_2),\xi_1 \rangle = \|H\|\cos\gamma + \frac{1}{8\|H\|}(\cos\gamma+\nu\sin\gamma)(\cos\alpha-\cos\beta),\\
& \langle h(X_2,X_2),\xi_2 \rangle = \|H\|\sin\gamma + \frac{1}{8\|H\|}(\sin\gamma-\nu\cos\gamma)(\cos\alpha-\cos\beta).
\end{aligned}
\end{equation}
\end{remark}

The following lemma gives a functional relation between the local functions $\alpha$, $\beta$, $\gamma$ and $\nu$, introduced in Lemma \ref{FrameLemma} and Lemma \ref{lem:h}.

\begin{lemma}\label{nuComp}
Let $\Phi: \Sigma \to \sh$ be a PMC surface such that the Hopf differential \eqref{eq:Hopf} vanishes and use the notations of Lemma \ref{FrameLemma} and Lemma \ref{lem:h}. Then, at all points where $\alpha$, $\beta$, $\gamma$ and $\nu$ are defined, at least one of the following holds: 
\begin{itemize}
\item[(i)] $\sin\alpha = 0$,
\item[(ii)] $\sin\beta = 0$,
\item[(iii)] $8\norm{H}^2 + (\cos\alpha-\cos\beta)((1-\nu^2)\cos(2\gamma)+2\nu\sin(2\gamma)) = 0$.
\end{itemize}
\end{lemma}

\begin{proof}
The proof consists of simplifying 
\begin{equation} \label{eq:compatnu}
X_1(X_2(\nu)) - X_2(X_1(\nu)) = (\nabla_{X_1}X_2-\nabla_{X_2}X_1)(\nu).
\end{equation}
On one hand, from \eqref{eq:Xalpha}, \eqref{eq:Xbeta}, \eqref{eq:Xgamma} and \eqref{eq:Xnu}, one can calculate
\begin{multline*}
X_1(X_2(\nu)) = -\frac{4\norm{H} \sin\alpha \sin\beta}{(\cos\alpha-\cos\beta)^2} (4\norm{H} \sin\gamma(\nu\cos\gamma-\sin\gamma) \\
+ \langle h(X_1,X_1) , \xi_2 \rangle (\nu \cos\gamma - \sin\gamma) + 2 \langle h(X_1,X_1) , \xi_1 \rangle (\nu \sin\gamma + \cos\gamma)),
\end{multline*}
\vspace{-.6cm}
\begin{multline*}
X_2(X_1(\nu)) = - \frac{4\norm{H} \sin\alpha \sin\beta}{(\cos\alpha-\cos\beta)^2} (4\norm{H} \cos\gamma(\nu\sin\gamma+\cos\gamma) \\
+ \langle h(X_2,X_2) , \xi_1 \rangle (\nu\sin\gamma+\cos\gamma) + 2 \langle h(X_2,X_2) , \xi_2 \rangle (\nu\cos\gamma-\sin\gamma)).
\end{multline*}
On the other hand, it follows from \eqref{eq:nablaX1X2} that
$$ \con{X_1}X_2 - \con{X_2}X_1 = -\frac{\sin\beta}{\cos\alpha-\cos\beta} \langle h(X_1,X_1),\xi_2 \rangle X_1 - \frac{\sin\alpha}{\cos\alpha-\cos\beta} \langle h(X_2,X_2),\xi_1 \rangle X_2, $$
and thus, by \eqref{eq:Xnu},
\begin{multline*} 
(\con{X_1}X_2 - \con{X_2}X_1)(\nu) \\ = \frac{4\norm{H}\sin\alpha\sin\beta}{(\cos\alpha-\cos\beta)^2}(\langle h(X_1,X_1),\xi_2 \rangle (\nu\cos\gamma - \sin\gamma) - \langle h(X_2,X_2),\xi_1 \rangle (\nu\sin\gamma + \cos\gamma)).
\end{multline*}
	
Note that if $\sin\alpha=0$ or $\sin\beta=0$, then $X_1(X_2(\nu)) = X_2(X_1(\nu)) = (\con{X_1}X_2 - \con{X_2}X_1)(\nu) = 0$. If $\sin\alpha \neq 0$ and $\sin\beta\neq 0$, then \eqref{eq:compatnu} is satisfied if and only if
\begin{align*}
2\norm{H} = (&\nu\sin\gamma+\cos\gamma)(\langle h(X_1,X_1),\xi_1 \rangle - \langle h(X_2,X_2),\xi_1 \rangle) \\ 
&+ (\nu\cos\gamma-\sin\gamma)(\langle h(X_1,X_1),\xi_2 \rangle - \langle h(X_2,X_2),\xi_2 \rangle).
\end{align*}
Using \eqref{eq:hXiXixij}, one can check that this equality is equivalent to condition (iii) in the statement.
\end{proof}

\begin{remark} \label{rem:interpretationSin0}
There is a geometric interpretation for the cases (i) and (ii) of Lemma \ref{nuComp}. From \eqref{eq:defF} and \eqref{eq:matrixF} it is clear that they mean that $X_1$, respectively $X_2$, lies in either $T\bb{S}^2\times \lbrace 0 \rbrace$ or $\lbrace 0 \rbrace \times T\bb{H}^2$. In the remainder of this section, it will be proven that case (iii) of Lemma \ref{nuComp} is redundant, that is, that the situation just described holds for any PMC surface with vanishing in Hopf differential in $\sh$.
\end{remark}

\begin{lemma}\label{gammaConstant}
Let $\Phi: \Sigma \to \sh$ be a PMC surface such that the Hopf differential \eqref{eq:Hopf} vanishes and use the notations of Lemma \ref{FrameLemma} and Lemma \ref{lem:h}. On a connected open subset of $\Sigma$ where $\alpha$, $\beta$ and $\gamma$ are defined, and $\sin\alpha \sin\beta$ nowhere vanishes, $\gamma$ is constant.
\end{lemma}

\begin{proof}
Let $U \subseteq \Sigma$ be an open subset as in the statement. Then equality (iii) from Lemma~\ref{nuComp} has to hold everywhere on $U$. By differentiating this equality in the directions of $X_1$ and $X_2$ respectively, and using \eqref{eq:Xalpha}, \eqref{eq:Xbeta}, \eqref{eq:Xgamma} and \eqref{eq:Xnu}, one obtains
\begin{equation}
\begin{aligned}
&((1-\nu^2)\cos(2\gamma)+2\nu\sin(2\gamma))\ip{h(X_1,X_1)}{\xi_1} \\ & - (2\nu\cos(2\gamma) -(1-\nu^2)\sin(2\gamma))\ip{h(X_1,X_1)}{\xi_2} \\
& + 4 \norm{H} (\sin\gamma - \nu \cos\gamma)(\sin(2\gamma) - \nu\cos(2\gamma)) = 0,
\end{aligned} \label{eq:X1(iii)}
\end{equation}
\begin{equation}
\begin{aligned}
& (2\nu\cos(2\gamma)-(1-\nu^2)\sin(2\gamma))\ip{h(X_2,X_2)}{\xi_1} \\
& + ((1-\nu^2)\cos(2\gamma)+2\nu\sin(2\gamma))\ip{h(X_2,X_2)}{\xi_2} \\ 
& - 4 \norm{H} (\cos\gamma + \nu \sin\gamma)(\sin(2\gamma) - \nu\cos(2\gamma)) = 0.
\end{aligned} \label{eq:X2(iii)}
\end{equation}

In combination with \eqref{eq:hXiXixij}, this implies that equation (iii) of Lemma \ref{nuComp} and equations \eqref{eq:X1(iii)} and \eqref{eq:X2(iii)} give rise to a system of the form
\begin{equation*}
\begin{cases}
A_0(\gamma,\nu)(\cos\alpha-\cos\beta) + 8\norm{H}^2 = 0, \\
A_1(\gamma,\nu)(\cos\alpha-\cos\beta) + B_1(\gamma,\nu) = 0, \\
A_2(\gamma,\nu)(\cos\alpha-\cos\beta) + B_2(\gamma,\nu) = 0,
\end{cases}
\end{equation*}
where $A_j(\gamma,\nu)$ and $B_j(\gamma,\nu)$ are polynomials in $\cos\gamma$, $\sin\gamma$ and $\nu$. For this system to have a solution for $\cos\alpha-\cos\beta$, it is necessary that
\begin{equation}\label{eq:systemfornu}
\begin{cases}
8\norm{H}^2 A_1(\gamma,\nu) = A_0(\gamma,\nu)B_1(\gamma,\nu), \\
8\norm{H}^2 A_2(\gamma,\nu) = A_0(\gamma,\nu)B_2(\gamma,\nu).
\end{cases}
\end{equation}

Now look at system \eqref{eq:systemfornu} as a system of polynomial equations in $\nu$ and consider its resultant $R(\cos\gamma,\sin\gamma)$, which is a polynomial function of $\cos\gamma$ and $\sin\gamma$. If $R$ is not the zero function, the possible values for $\gamma$ come from a discrete set, which implies that $\gamma$ is constant on the connected subset $U \subseteq \Sigma$. Therefore, to finish the proof, it suffices to exclude the case that $R$ is the zero function.

Assume, by contradiction, that $R$ is identically zero. This implies that \eqref{eq:systemfornu} has a solution for every value of $\gamma$. One can check that this is not true for $\gamma=0$. A direct computation shows that 
$$ A_0(0,\nu) = 1-\nu^2, \qquad A_1(0,\nu) = -\frac{1+\nu^2}{8\norm{H}}, \qquad A_2(0,\nu) = -\frac{\nu(1+\nu^2)}{8\norm{H}}, $$
$$ B_1(0,\nu) = \norm{H}(1 + 3\nu^2), \qquad  B_2(0,\nu) = -6\norm{H}\nu. $$
Therefore, system \eqref{eq:systemfornu}, in the case $\gamma=0$, is equivalent to
$$ \begin{cases}
-(1+\nu^2)=(1-\nu^2)(1+3\nu^2), \\
\nu(1+\nu^2)=6\nu(1-\nu^2).	
\end{cases}$$
It is clear that these two polynomial equations do not have a common solution.
\end{proof}

\begin{theorem}\label{nonTrivialIntersection}
Let $\Phi\colon\Sigma\to \sh$ be a PMC surface such that the Hopf differential \eqref{eq:Hopf} vanishes. Then, for all $p\in \Sigma$, the intersection of $T_p\Sigma$ with either $T_{\pi_1(p)}\bb{S}^2 \times \{0\}$ or $\{0\} \times T_{\pi_2(p)}\bb{H}^2$ is non-trivial.
\end{theorem} 
\begin{proof}
First note that it is sufficient to prove the result for an open and dense subset of $\Sigma$. One way to see this is noting that $T_p\Sigma \cap (T_{\pi_1(p)}\bb{S}^2 \times \{0\})$ or $T_p\Sigma \cap (\{0\} \times T_{\pi_2(p)}\bb{H}^2)$ is non-trivial if and only if the  determinant of the derivative of $\pi_1\circ\Phi$ or of $\pi_2\circ\Phi$ at $p$ vanishes. If, at every point of an open and dense subset of $\Sigma$, one of these determinants vanishes, this must also be the case on its complement by continuity.

The result will be proven for the open and dense subset $\Sigma_0 \subseteq \Sigma$ of points where $\Phi$ is not pseudo-umbilical. Take  $p\in\Sigma_0$ and define the frame $(X_1,X_2,\xi_1,\xi_2)$ and the functions $\alpha$, $\beta$, $\gamma$ and $\nu$ in a neighborhood of $p$ as in Lemma \ref{FrameLemma} and Lemma \ref{lem:h}. As noted in Remark \ref{rem:interpretationSin0}, it is sufficient to show that $\sin\alpha(p)=0$ or $\sin\beta(p)=0$.

Assume, by contradiction, that $\sin\alpha(p)\sin\beta(p) \neq 0$. By continuity, there is a neighborhood $U$ of $p$ in which $\sin\alpha\sin\beta$ does not vanish and, by Lemma \ref{gammaConstant}, it follows that $\gamma$ is constant in~$U$. By \eqref{eq:Xgamma}, $X_1(\gamma)=X_2(\gamma)=0$ implies $\langle h(X_1,X_1),\xi_2 \rangle = \langle h(X_2,X_2),\xi_1 \rangle = 0$, which, by \eqref{eq:hXiXixij}, is equivalent to
$$\begin{cases}
\displaystyle{\|H\|\sin\gamma - \frac{1}{8\|H\|}(\sin\gamma-\nu\cos\gamma)(\cos\alpha-\cos\beta) = 0,} \\
\displaystyle{\|H\|\cos\gamma + \frac{1}{8\|H\|}(\cos\gamma+\nu\sin\gamma)(\cos\alpha-\cos\beta) = 0.} \\ 
\end{cases}$$
Looking at these equations as a system in $\cos\alpha-\cos\beta$, the determinant of the system must vanish, which is equivalent to $\sin(2\gamma)-\nu\cos(2\gamma)=0$. Since $\gamma$ is constant, this implies that $\nu$ is constant. By \eqref{eq:Xnu}, $X_1(\nu)=X_2(\nu)=0$ implies 
$$\begin{cases}
\sin\gamma-\nu\cos\gamma=0, \\
\cos\gamma+\nu\sin\gamma=0.
\end{cases}$$
This can be seen as a system of equations for $\nu$ with determinant $\sin^2\gamma+\cos^2\gamma=1\neq0$, so it cannot have a solution for $\nu$. This is the contradiction that finishes the proof.
\end{proof}

\begin{remark}\label{inclusions}
Note that we have the following nested families of PMC surfaces of $\sh$: 
\begin{multline*}
\lbrace \text{PMC spheres} \rbrace \subseteq \lbrace \text{PMC surfaces with vanishing Hopf differential } \Theta \rbrace \\
\subseteq \left\lbrace \begin{array}{c} \text{PMC surfaces } \Sigma \text{ such that for every }p\in\Sigma,~T_p\Sigma \text{ intersects } \\ T_{\pi_1(p)}\bb{S}^2\times \lbrace 0 \rbrace \text{ or } \lbrace 0 \rbrace\times T_{\pi_2(p)}\bb{H}^2 \text{ non-trivially} \end{array} \right\rbrace.
\end{multline*}
\end{remark}
In the remainder of the text, we give a classification result for all three families mentioned in the remark above. The following Lemma characterizes the last family in terms of the K\"ahler functions $C_1$ and $C_2$ defined in equation \eqref{eq:defKahler}.

\begin{lemma}\label{Csquared}
	Let $\Phi\colon \Sigma \to \sh$ be a surface of $\sh$. For any point $p\in \Sigma,~T_p\Sigma$ intersects $T_{\pi_1(p)}\bb{S}^2\times \lbrace 0 \rbrace$ or $\lbrace 0 \rbrace\times T_{\pi_2(p)}\bb{H}^2$ non-trivially if and only if the K\"ahler functions $C_1$ and $C_2$ satisfy $C_1^2(p) = C_2^2(p)$. Moreover, if $\Phi$ is a PMC immersion and $T_p\Sigma$ intersects $T_{\pi_1(p)}\bb{S}^2\times \lbrace 0 \rbrace$ or $\lbrace 0 \rbrace\times T_{\pi_2(p)}\bb{H}^2$ non-trivially for every point $p\in \Sigma$, then either $C_1 = C_2$ everywhere or $C_1 = -C_2$ everywhere.
\end{lemma}

\begin{proof}
	Denote the component functions of $\Phi$ by $\phi= \pi_1\circ\Phi$ and $\psi= \pi_2 \circ \Phi$. Let $\omega_{\bb{S}^2},\omega_{\H^2}$ and $\omega_{\Sigma}$ denote the area 2-forms of $\mathbb{S}^2,\mathbb{H}^2$ and $\Sigma$, respectively. Let $\omega_1$ and $\omega_2$ denote the K\"ahler 2-forms of $J_1$ and $J_2$ respectively. As found in \cite{TorralboUrbano}, these forms satisfy $\omega_1 = \pi_1^* \omega_{\bb{S}^2} + \pi_2^* \omega_{\H^2} \quad \text{and} \quad \omega_2 = \pi_1^* \omega_{\bb{S}^2} - \pi_2^* \omega_{\H^2}$. For any pair of vectors $e_1,e_2\in T_p\Sigma$, one can then use Definition \eqref{eq:defKahler} to find that
	\begin{align*}
		\omega_{\bb{S}^2} (\phi_* e_1, \phi_* e_2)
		&= \omega_{\bb{S}^2} ((\pi_1)_* \Phi_* e_1, (\pi_1)_* \Phi_* e_2)
		= (\pi_1^* \omega_\bb{S}^2)(\Phi_* e_1, \Phi_* e_2) \\
		&= \frac 12 (\omega_1(\Phi_* e_1, \Phi_* e_2) + \omega_2(\Phi_* e_1, \Phi_* e_2)) 
		= \frac 12 ((\Phi^* \omega_1) (e_1,e_2) + (\Phi^* \omega_2) (e_1,e_2)) \\
		&= \frac{C_1+C_2}2\omega_\Sigma(e_1,e_2).
	\end{align*}
	Thus, $\phi^* \omega_{\mathbb{S}^2} = \frac{C_1+C_2}2\omega_\Sigma$ and one can show analogously that $\psi^* \omega_{\mathbb{H}^2} = \frac{C_1-C_2}2\omega_\Sigma$. This is mentioned in Section 3 of \cite{TorralboUrbano}, it is phrased there as $\text{Jac}(\phi) = \frac{C_1+C_2}2$ and $\text{Jac}(\psi) = \frac{C_1-C_2}2$.
	
	It is clear that $T_p\Sigma$ intersects $T_{\pi_1(p)}\bb{S}^2\times \lbrace 0 \rbrace$ or $\lbrace 0 \rbrace\times T_{\pi_2(p)}\bb{H}^2$ non-trivially if and only if the total derivative $d\psi_p$, respectively $d\phi_p$ is singular. By the previous calculation, this is equivalent to the property that $C_1^2(p) = C_2^2(p)$.
	
	It remains to show the final claim. Assume that $\Phi$ is a PMC immersion such that $T_p\Sigma$ intersects $T_{\pi_1(p)}\bb{S}^2\times \lbrace 0 \rbrace$ or $\lbrace 0 \rbrace\times T_{\pi_2(p)}\bb{H}^2$ non-trivially for every point $p\in \Sigma$. Similar to \cite{TorralboUrbano} equation (3.8), one can use Equations \eqref{eq:DCj&Dgammaj} and \eqref{eq:Dfj} to find that the Laplacian of $C_1+C_2$ is
	$$\Delta (C_1 + C_2) = (-2|\gamma|^2 - 16e^{-2u}|f|^2 -2\norm{H}^2e^{2u})(C_1+C_2),$$
	where $|\gamma| = |\gamma_1| = |\gamma_2|$ and $|f| = |f_1| = |f_2|$.
	Thus, if $C_1 + C_2 = 0$ does not hold everywhere, then Theorem 2.2 in \cite{NodalSets} says the zero set $\lbrace p\in \Sigma \mid C_1(p)+C_2(p) = 0 \rbrace$ forms a 1-dimensional submanifold, except at a discrete set of points. In particular, its complement is an open dense set where $C_1 = C_2$. In this case, $C_1 = C_2$ holds everywhere by continuity.
\end{proof}

\section{Examples}\label{examples}

By combining Proposition \ref{TrivialPMC} and Proposition \ref{CMCtu}, one can construct an example of a PMC surface in $\sh$, starting from any CMC surface in $\mathbb S^2\times\R$ or $\mathbb H^2\times\R$. In this section, we will study two more types of examples which have analogues in $\bb{S}^2\times \bb{S}^2$ and $\bb{H}^2 \times \bb{H}^2$, see \cite{TorralboUrbano}. Each of these surfaces will have the property that at every point $p$, the tangent space of the surface intersects $T_{\pi_1(p)}\bb{S}^2\times \lbrace 0 \rbrace$ or $\lbrace 0 \rbrace\times T_{\pi_2(p)}\bb{H}^2$ non-trivially. We also study the Hopf differential for these examples, as this is useful in later classifications. 

First, we study the surfaces arising from Proposition \ref{TrivialPMC} and Proposition \ref{CMCtu}. In \cite{AbrRos}, Abresch and Rosenberg defined a holomorphic quadratic differential for CMC surfaces in $\bb{S}^2\times \bb{R}$ and $\bb{H}^2\times \bb{R}$. We denote this differential by $\Theta_{\text{AR}}$. We will relate the Hopf differential $\Theta$ from Lemma~\ref{Hopf} to $\Theta_{\text{AR}}$.

\begin{proposition}\label{trivialExamples}
Let $\Phi_1\colon\Sigma\to M^2(\epsilon)\times \bb{R}$ be a surface with non-zero constant mean curvature, where $\epsilon\in\lbrace -1,1\rbrace$. Let $\Phi_2\colon M^2(\epsilon)\times \bb{R}\to \sh$ be a totally geodesic isometric immersion. Then $\Phi=\Phi_2 \circ \Phi_1$ is a PMC surface in $\sh$. Furthermore, the Hopf differential \eqref{eq:Hopf} satisfies $\Theta = 4 \Theta_{\text{AR}}$, where $\Theta_{\text{AR}}$ is the differential for $\Phi_1$ defined in \cite{AbrRos}.
\end{proposition}

\begin{proof}
This is analogous to Lemma 1 in \cite{TorralboUrbano}, only minor changes are needed to the proof.
\end{proof}

The next example is a product of curves, which is the analogue of Example 1 in \cite{TorralboUrbano}. 

\begin{proposition}\label{curveProduct}
Let $I,I'\subseteq\bb{R}$ be intervals. Let $\alpha\colon I\to \bb{S}^2$ and $\beta\colon I' \to \bb{H}^2$ be regular curves with constant curvatures $k_\alpha$ and $k_\beta$, not both zero. Then
$$\Phi\colon I\times I'\to \sh\colon (t,s)\mapsto (\alpha(t),\beta(s))$$
is a PMC surface whose mean curvature vector satisfies $\norm{H}^2 = \frac{1}{4}(k_\alpha^2 + k_\beta^2)$. The Hopf differential \eqref{eq:Hopf} vanishes if and only if $k_\beta^2 = 1+k_\alpha^2$.
\end{proposition}

\begin{proof}
Arguing that these surfaces are PMC is similar to \cite{TorralboUrbano}, Example 1. It remains to determine the differential $\Theta$. One can assume, after a reparametrization if necessary, that $\norm{\alpha'}^2 = \norm{\beta'}^2=1$. Then $z = t+is$ is a (global) conformal parameter on the surface. Note that $\partial_t = \alpha'$ and $\partial_s = \beta'$. Since the curves have unit speed, the second fundamental form then satisfies $h\left(\partial_t, \partial_t \right) = \nabla^{\bb{S}^2}_{\alpha'}\alpha'$ and $h\left(\partial_s, \partial_s \right) = \nabla^{\bb{H}^2}_{\beta'}\beta'$, while $h\left(\partial_t, \partial_s \right) = 0$. Hence, $H = \frac 12 (\nabla^{\bb{S}^2}_{\alpha'}\alpha' + \nabla^{\bb{H}^2}_{\beta'}\beta')$. Furthermore, $F\dz = \dzbar$. One can thus calculate directly that $\Theta = \left(\frac{1}{2}(k_\alpha^2-k_\beta^2) + \frac{1}{2}\right) dz\otimes dz$.
\end{proof}

The final examples we want to discuss are more involved. They are analogues to the examples in $M^2(\epsilon)\times M^2(\epsilon)$, discussed in Proposition 5 in \cite{TorralboUrbano}.

\begin{proposition}\label{SpecialSurfaces1}
Let $a,b,c\in \bb{R}$ be real constants such that $a>0$ and $b>0$. Let $h\colon I\subseteq \bb{R}\to \bb{R}$ be a non-constant function satisfying
\begin{equation}\label{hEquation1} 
(h'(x))^2 = \left(a-h(x)^2 - b(1+(h(x)-c)^2) \right)\left( a - h(x)^2 \right)
\end{equation}
and $a-h(x)^2 > 0$ for all $x \in I$. Define the map
$$\phi\colon I \times \bb{R} \to \bb{S}^2\colon (x,y)\mapsto \frac{1}{\sqrt{a}} \left(\sqrt{a-h(x)^2}\cos\left(\sqrt{a}y \right), \sqrt{a-h(x)^2}\sin\left(\sqrt{a}y \right),h(x)\right)$$
and a curve $\psi: I \to \bb{H}^2$ with velocity $\norm{\psi'} = \sqrt{b(1+(h-c)^2)}$ and curvature $\kappa_{\psi} = -\frac{b(a-h^2)}{\norm{\psi'}^3}$. Then
$$\Phi\colon I\times \bb{R}\to \sh\colon (x,y)\mapsto (\phi(x,y),\psi(x)) $$
is a PMC surface in $\sh$ whose mean curvature satisfies $\norm{H}^2 = \frac b4$. The Hopf differential~\eqref{eq:Hopf} vanishes if and only if $a = 1+c^2$.
\end{proposition}
\begin{proof}
	Showing that these are PMC surfaces is an analogous calculation as in the proof of Proposition 5 in \cite{TorralboUrbano}. In particular, one can first calculate that $\norm{\phi_y}^2 = (a-h^2)$ and $\norm{\phi_x}^2 = a-h^2-b\left(1+(h-c)^2\right)$, while $\ip{\phi_x}{\phi_y}=0$. Further, $\psi_y = 0$ and $\norm{\psi_x}^2 = b\left(1+(h-c)^2\right)$ by definition. Thus, $( x,y)$ are conformal coordinates with conformal factor $a-h^2$. Finally, one can find that $4\norm{H}^2 = b$ and
	$$\conM{\dx} H  = -\frac{b(a-ch)}{2(a-h^2)} \Phi_x, \qquad \conM{\dy}H = \frac{bh(h-c)}{2(a-h^2)}\Phi_y.$$ \\ It remains to show the final statement, namely that $\Theta \equiv 0$ if and only if $a = 1+c^2$.
	By direct calculation, the first term of the Hopf differential is $4\ip{H}{h(\dz,\dz)} = 4\ip{-\conM{\dz}H}{\dz} = \frac{1}{2}(ab-2bch+bh^2)$
	and the second term is
		$\ip{F\dz}{\dz}
		= \frac{1}{4}\left(\norm{\phi_x}^2 - \norm{\phi_y}^2 - \norm{\psi_x}^2\right)
		= \frac{1}{2}\left(-b - bh^2 - bc^2 + 2bch \right).$
	Therefore, the differential $\Theta$ vanishes if and only if
		$0 = 4\ip{h(\dz,\dz)}{H} + \ip{F\dz}{\dz} = \frac{b}{2}\left(a-1-c^2 \right),$
	which is equivalent to $a=1+c^2$, as $b>0$ by assumption.
\end{proof}
\begin{proposition}\label{SpecialSurfaces2}
Let $a,b,c\in \bb{R}$ be real constants such that $b>0$. Let $h\colon I\subseteq \bb{R}\to \bb{R}$ be a non-constant function satisfying
\begin{equation}\label{hEquation2} (h'(x))^2 = \left(a-h(x)^2 + b(1+(h(x)-c)^2) \right)\left( a - h(x)^2 \right)\end{equation}
and $h(x)^2-a > 0$ for all $x \in I$. Define the map
\begin{multline*}\psi\colon I \times \bb{R} \to \bb{H}^2\colon \\ 
(x,y)\mapsto
\begin{cases}
\frac{1}{\sqrt{a}} \left(\sqrt{h(x)^2-a}\cos\left(\sqrt{a}y \right), \sqrt{h(x)^2-a}\sin\left(\sqrt{a}y \right),h(x)\right) & \text{if } a>0, \\
\frac{1}{\sqrt{-a}}\left(h(x),\sqrt{h(x)^2-a}\sinh\left(\sqrt{-a}y \right), \sqrt{h(x)^2-a}\cosh\left(\sqrt{-a}y \right)\right) & \text{if } a<0, \\
\frac{1}{2h(x)}\left((y^2-1)h(x)^2+1,2yh(x)^2, (y^2+1)h(x)^2+1 \right) & \text{if } a=0
\end{cases}
\end{multline*}
and a curve $\phi: I \to \bb{S}^2$ with velocity $\norm{\phi'} = \sqrt{b(1+(h-c)^2)}$ and curvature is $\kappa_{\phi} = -\frac{b(h^2-a)}{\norm{\phi'}^3}$. Then
$$\Phi\colon I\times \bb{R}\to \sh\colon(x,y)\mapsto (\phi(x),\psi(x,y)) $$
is a PMC surface in $\sh$ whose mean curvature vector satisfies $\norm{H}^2 = \frac b4$. The Hopf differential \eqref{Hopf} vanishes if and only if $a = 1+c^2$.
\end{proposition}

\begin{proof}
Calculating that these surfaces are PMC is again analogous to Proposition 5 in \cite{TorralboUrbano}. Proving that the Hopf differential \eqref{eq:Hopf} vanishes if and only if $a = 1+c^2$ is done similarly as in the proof of Proposition \ref{SpecialSurfaces1} above.
\end{proof}

\begin{remark}
	For a detailed discussion of the differential equations \eqref{hEquation1} and \eqref{hEquation2}, see Section 4 of \cite{TorralboUrbano}.
\end{remark}

\section{Classifications}\label{classification}

In \cite{TorralboUrbano}, the PMC surfaces in $M^2(\epsilon)\times M^2(\epsilon)$ for which the K\"ahler functions $C_1$ and $C_2$ satisfy $C_1^2=C_2^2$ are classified. Subsequently, the PMC surfaces with vanishing Hopf differentials and the PMC spheres are classified. The setting is analogous to our Remark \ref{inclusions}, after using Lemma \ref{Csquared}. Therefore, we use similar techniques to classify PMC surfaces with $C_1^2 = C_2^2$ in $\sh$.

The PMC spheres can be classified in the same way as in \cite{TorralboUrbano}.

\begin{theorem}\label{PMCspheres}
	Let $\Sigma$ be a surface which is topologically a sphere. Then $\Phi\colon\Sigma\to \sh$ is a PMC sphere if and only if $\Phi$ is congruent to a CMC sphere of a totally geodesic $\bb{S}^2\times \bb{R}$ or $\bb{H}^2\times \bb{R}$.
\end{theorem}
\begin{proof}
	See the first part of the proof of Theorem 3 in \cite{TorralboUrbano}.
\end{proof}

After some adaptations, one can follow the proof of Theorem 3 in \cite{TorralboUrbano} to find a classification for the PMC surfaces $\Sigma$ such that at every point $p\in \Sigma$, $T_p\Sigma$ intersects $T_{\pi_1(p)}\bb{S}^2\times \lbrace 0 \rbrace$ or $\lbrace 0 \rbrace\times T_{\pi_2(p)}\bb{H}^2$ non-trivially.
\begin{theorem}\label{MainClassification}
	Let $\Phi\colon \Sigma \to \sh$ be a PMC surface. Then $\Sigma$ has the property that at every point $p\in \Sigma$, $T_p\Sigma$ intersects $T_{\pi_1(p)}\bb{S}^2\times \lbrace 0 \rbrace$ or $\lbrace 0 \rbrace\times T_{\pi_2(p)}\bb{H}^2$ non-trivially if and only if $\Phi$ is locally congruent to
	\begin{enumerate}
		\item A CMC surface of $\bb{S}^2\times \bb{R}$ or $\bb{H}^2\times \bb{R}$, see Proposition \ref{trivialExamples}
		\item A product of curves with constant curvature, see Proposition \ref{curveProduct}
		\item One of the Examples from Propositions \ref{SpecialSurfaces1} or \ref{SpecialSurfaces2}.
	\end{enumerate}
\end{theorem}
\begin{proof}
	We only give a brief explanation, as the majority of the proof is essentially the same as that of \cite{TorralboUrbano}, Theorem 3. Particularly, the derivation for the derivative of the conformal factor $e^u$ doesn't readily generalize to our setting, as it depends on the two Hopf differentials. However, one can obtain the same equation from different techniques, as will be explained below.
	
	Assume that $\Sigma$ is orientable by using the two-fold oriented cover if necessary.
	
	By Lemma \ref{Csquared}, there are 2 cases, namely the case where $C_1 = C_2$ and the case where $C_1 = -C_2$ everywhere. We consider $C_1 = C_2$, as $C_1 = -C_2$ is analogous.
	
	The case that requires a modification from \cite{TorralboUrbano} is when the holomorphic differential $\Omega = (\gamma_1 - \gamma_2)dz$ is non-vanishing.
	Normalize $\Omega$ to get $\gamma_2 - \gamma_1 = 2\sqrt{2}\norm{H}$. By equation \eqref{eq:gammanorm}, the norms of $\gamma_1$ and $\gamma_2$ satisfy $|\gamma_1| = |\gamma_2|$, so that there is a real function $g\colon\Sigma\to \bb{R}$ such that
	$$\gamma_1 = -\sqrt{2}\norm{H} + ig, \qquad  \gamma_2 = \sqrt{2}\norm{H} + ig.$$
	Like in \cite{TorralboUrbano}, the function $g$, the K\"ahler functions $C_j$ and the conformal factor $e^{2u}$ are independent of the $y$ coordinate, and $C_1(\bar{f}_1-f_2) = 0$.
	Hence, either $C_1 = C_2 = 0$ in an open neighborhood or the zeros of $C_1$ are isolated and then $\bar{f}_1 = f_2$ everywhere.
	In case $C_1 = C_2 = 0$, $\Sigma$ is a product of curves, so it is an example from Proposition \ref{curveProduct}.
	
	In the other case, the equation $\bar{f}_1 = f_2$ holds everywhere. Our aim is to show that there is a constant $\tilde{\mu}$ such that $u'$, the $x$-derivative of $u$, satisfies $u' = \left(\frac{g}{\sqrt{2}\norm{H}} + \tilde{\mu} \right)C_1$, which is the part of the proof of \cite[Theorem 3]{TorralboUrbano} that does not easily carry over to our setting. Remark that by equations \eqref{eq:Dfj},
	$(f_1)_z = \overline{(\bar{f}_1)_{\zbar}} = \overline{(f_2)_{\zbar}} = -i\frac{e^{2u}}{4}C_1\bar{\gamma}_2 = i\frac{e^{2u}}{4} C_2\gamma_1 = (f_1)_{\zbar}. $
	Thus, $f_1$ and $f_2$ are independent of the $y$ coordinate. Write the $x$-derivative as $'$. Further, write $f_1 = l + im$, where $l,m\colon\Sigma\to \bb{R}$ are real functions. Then the $x$-derivative of $f_1$ is
		$$l'+im' = f_1' = 2(f_1)_{\zbar} = -i\frac{\norm{H}}{\sqrt{2}}e^{2u}C_1-\frac{e^{2u}C_1g}{2} ,$$
	while the $x$-derivative of $g$ is
		$$g' = -i\gamma_1' = -2i(\gamma_1)_{\zbar} = -\sqrt{2}\norm{H}e^{2u}C_1.$$
	The above equations imply that $2m' = g'$. Therefore, there is a real constant $\mu$ such that $m = \frac{g}{2} + \mu$. By using equations \eqref{eq:DCj&Dgammaj} and considering the real part of $(\gamma_j)_z = (\gamma_j)_{\zbar}$, one can then calculate that
	$$u' = \left(\frac{g}{\sqrt{2}\norm{H}} + \tilde{\mu} \right)C_1,$$
	where $\tilde{\mu} = \sqrt{2}\mu/\|H\|$, which is constant.
	From this point onward, it is straightforward to replicate the integration of the Frenet equation as done in \cite[Theorem 3]{TorralboUrbano}. In case $C_1 = C_2$, one finds the examples from Proposition \ref{SpecialSurfaces1}. In case $C_1 = -C_2$, one finds the ones from Proposition \ref{SpecialSurfaces2} instead.
\end{proof}
One can now use Theorem \ref{nonTrivialIntersection} to find a classification of the PMC surfaces with vanishing Hopf differential $\Theta$.
\begin{corollary}\label{trivialHopf}
	Let $\Phi\colon \Sigma \to \sh$ be a PMC surface. The Hopf differential $\Theta$ vanishes if and only if $\Phi$ is locally congruent to
	\begin{enumerate}
		\item A CMC surface of $\bb{S}^2\times \bb{R}$ or $\bb{H}^2\times \bb{R}$ whose Abresch-Rosenberg differential vanishes, see Proposition \ref{trivialExamples}
		\item A product of curves $\alpha\colon I\subseteq\bb{R}\to\bb{S}^2$ and $\beta\colon I'\subseteq\bb{R}\to\bb{H}^2$ with constant curvatures $k_\alpha,k_\beta$ satisfying $k_\beta^2 = 1+k_\alpha^2$, see Proposition \ref{curveProduct}
		\item One of the Examples from Propositions \ref{SpecialSurfaces1} or \ref{SpecialSurfaces2} which satisfies $a = 1+c^2$.
	\end{enumerate}
\end{corollary}
\begin{proof}
	This is a consequence of Theorem \ref{MainClassification} after using Theorem \ref{nonTrivialIntersection}. The differentials were discussed in Propositions \ref{trivialExamples}--\ref{SpecialSurfaces2}.
\end{proof}

One can also determine the Lagrangian PMC surfaces of $\sh$. The result is analogous to Theorem 2 in \cite{TorralboUrbano}, but the proof is remarkably easier. This is because of equations \eqref{eq:Dfj}, which are different from those in \cite{TorralboUrbano}. Particularly, in our setting, $C_2$ appears as a factor in $\dzbar f_1$.

\begin{proposition}
	An immersion of a surface $\Phi\colon \Sigma \to \sh$ is PMC and Lagrangian with respect to one of $J_1 = (J_{\mathbb{S}^2},J_{\mathbb{H}^2})$ or $J_2 = (J_{\mathbb{S}^2},-J_{\mathbb{H}^2})$ if and only if $\Phi$ is locally congruent to a product of curves with constant curvature.
\end{proposition}
\begin{proof}
	The surface being Lagrangian with respect to $J_j$ is equivalent to the vanishing of the K\"ahler function $C_j$. It is clear that a product of curves is Lagrangian with respect to both $J_1, J_2$. \\
	Conversely, suppose that a surface $\Phi\colon \Sigma \to \sh$ is Lagrangian with respect to one of the $J_j$, that is, one of the $C_j$ is vanishing. Assume that $C_1$ vanishes, the other case is similar. In case $C_1$ vanishes, so does its derivative. By equation \eqref{eq:DCj&Dgammaj}, $2\sqrt{2}e^{-2u}\overline{\gamma_1}f_1 = \norm{H}\gamma_1$. Multiplying this equation by $\gamma_1$ yields $f_1 = \frac{\norm{H}}{\sqrt{2}}\gamma_1^2$, as $|\gamma_1|^2 = \frac{e^{2u}}{2}(1-C_1^2) = \frac{e^{2u}}{2}$.
	Using equations \eqref{eq:DCj&Dgammaj} and \eqref{eq:Dfj}, it follows that $i\frac{e^{2u}}{4}\gamma_1 C_2 = 0$. Since $|\gamma_1|^2=e^{2u}/2$, this implies that $C_2 = 0$ as well. \\
	Thus, a PMC surface is Lagrangian with respect to one of the $J_j$ if and only if it is Lagrangian with respect to both. When $C_1,C_2=0$, then the proof of Lemma \ref{Csquared} implies $\phi^*\omega_{\mathbb{S}^2}$ and $\psi^*\omega_{\mathbb{H}^2}$ both vanish. Hence, the surface is a product of curves. 
\end{proof}

\end{document}